\documentclass[12pt]{amsart}
\usepackage{amssymb,verbatim,enumerate,ifthen}
\usepackage{cite}
\usepackage[mathscr]{eucal}
\usepackage[utf8]{inputenc}
\usepackage[T1]{fontenc}
\usepackage{esint}
\usepackage[marginparwidth=2.4cm]{geometry}

\newtheorem{thm}{Theorem}[section]
\newtheorem{cor}[thm]{Corollary}
\newtheorem{prop}[thm]{Proposition}
\newtheorem{lem}[thm]{Lemma}

\theoremstyle{definition}
\newtheorem{defin}[thm]{Definition}

\newtheorem{exa}[thm]{Example}

\numberwithin{equation}{section}
\def\eq#1{{\rm(\ref{#1})}}
\def\Eq#1#2{\ifthenelse{\equal{#1}{*}}
  {\begin{equation*}\begin{aligned}[]#2\end{aligned}\end{equation*}}
  {\begin{equation}\begin{aligned}[]\label{#1}#2\end{aligned}\end{equation}}}

\def\A{\mathscr{A}}

\def\D{\mathscr{D}}
\def\LD{\underline{\D}}
\def\UD{\overline{\D}}
\def\LLD{\underline{\underline{\D}}}

\def\UUD{\overline{\overline{\D}}}
\def\E{\mathscr{E}}
\def\LE{\underline{\E}}
\def\UE{\overline{\E}}
\def\LLE{\underline{\underline{\E}}}
\def\UUE{\overline{\overline{\E}}}

\def\M{\mathscr{M}}
\def\P{\mathscr{P}}
\def\Err{\mathcal{O}}
\newcommand\R{\mathbb{R}}

\newcommand\N{\mathbb{N}}

\newcommand\Q{\mathbb{Q}}
\newcommand{\QA}[1]{\A_{#1}}

\DeclareMathOperator{\sign}{sign}

\title
{On the homogenization of means}

\author{Zsolt P\'ales}
\address{Institute of Mathematics, University of Debrecen, Pf.\ 400, 4002 Debrecen, Hungary}
\email{pales@science.unideb.hu}

\author{Pawe\l{} Pasteczka}
\address{Institute of Mathematics, Pedagogical University of Cracow,  Podchor\k{a}\.{z}ych str 2, 30-084 Cracow, Poland}
\email{pawel.pasteczka@up.krakow.pl}

\thanks{The research of the first author was supported by the Hungarian Scientific Research Fund (OTKA) Grant K-111651 and by the EFOP-3.6.2-16-2017-00015 project. This project is co-financed by the European Union and the European Social Fund.}



\begin{document}
\begin{abstract}
The aim of this paper is to introduce several notions of homogenization in various classes of weighted means, which include quasiarithmetic and semideviation means. In general, the homogenization is an operator which attaches a homogeneous mean to a given one. Our results show that, under some regularity or convexity assumptions, the homogenization of quasiarithmetic means are power means, and homogenization of semideviation means are homogeneous semideviation means. In other results, we characterize the comparison inequality, the Jensen concavity, and Minkowski- and Hölder-type inequalities related to semideviation means.
\end{abstract}

\maketitle
\section{Introduction}

The investigation of means concentrates to two, somehow independent, subfields. The first subfield is related to the extensive study of homogeneous means which includes comparison and characterization problems, invariance equations and Minkow\-ski- and Hölder-type inequalities (cf.\ \cite{BajPal09b}, \cite{BajPal10}, \cite{BajPal13}, \cite{LosPal96}, \cite{LosPal98}, \cite{Pal82c}, \cite{Pal83c}, \cite{Pal88b}, \cite{Pal88c}, \cite{Pal89c}, \cite{Pal92a}). The postulated domain of homogeneous means is usually the set of positive numbers and the means considered are power means, Gini means, Stolarsky means, symmetric polynomial means, counter-harmonic means, etc. 

On the other hand, a rich literature is devoted to the investigation of means which are not necessarily homogeneous. What is important in that setting, is that the freedom in nonhomogeneous families is fairly bigger. Namely, the means considered therein are depending on functions instead of parameters (which is usually the case for homogeneous means). 

Perhaps the most classical illustration is the family of quasiarithmetic means and the related family of power means (cf.\ \cite{BecBel61}, \cite{Bul03}, \cite{HarLitPol34}, \cite{MitPecFin91}, \cite{MitPecFin93} and the references therein). The first family is defined in terms of functions, the second in terms of a real parameter. Nevertheless, homogeneity is sometimes very natural -- for example, in actuarial mathematics, it is known as ``currency invariance''. Therefore, it seems to be an important issue to investigate and clarify how functional means are related to homogeneous means.

In this paper we introduce two constructions for the homogenization of weighted means (Section~\ref{sec:homomean}). In Section~\ref{sec:examples} we will deal with the homogenizations related to  quasiarithmetic means. 

The main class, where we elaborate the notions related to homogenization are  semideviation means. They naturally extend the classes of deviation means and quasideviation means introduced by Daróczy \cite{Dar71b} and Páles \cite{Pal83b}, \cite{Pal88a}, respectively. Some of our main results offer characterizations of the comparison and of Hölder- and Minkowski-type inequalities for semideviation means. The results so obtained extend those of \cite{Dar72b} and \cite{Pal88a}.

As a matter of fact, our considerations are motivated by studying the weighted Hardy-type inequalities (see \cite{PalPas18a}, \cite{PalPas18b}). They turned out to be closely related to the Kedlaya-type inequalities \cite{Ked94}, \cite{Ked99}. On the other hand, the Hardy property is implicitly connected to homogeneity, furthermore in its characterizations, the Jensen-concavity of means plays a very important role. Motivated by these preliminaries, we provide a thorough study, among others, of these properties for different concepts of homogenizations.

%\cite{ChuMIA18}

\subsection{Weighted means}

To describe the abstract framework for our considerations, we recall the concept of weighted means as it was introduced in the paper \cite{PalPas18b}.

\begin{defin}[Weighted means]
Let $I \subset \R$ be an arbitrary interval, $R \subset \R$ be a ring and, for 
$n\in\N$, define the set of $n$-dimensional weight vectors $W_n(R)$ by
\Eq{*}{
  W_n(R):=\{(\lambda_1,\dots,\lambda_n)\in R^n\mid\lambda_1,\dots,\lambda_n\geq0,\,\lambda_1+\dots+\lambda_n>0\}.
}
\emph{A weighted mean on $I$ over $R$} or, in other words, \emph{an $R$-weighted mean on $I$} is a function 
\Eq{*}{
\M \colon \bigcup_{n=1}^{\infty} I^n \times W_n(R) \to I
}
satisfying the conditions (i)--(iv) presented below.
Elements belonging to $I$ will be called \emph{entries}; elements from $R$ -- \emph{weights}. 

\begin{enumerate}[(i)]
 \item \emph{Nullhomogeneity in the weights}: For all $n \in \N$, for all $(x,\lambda) \in I^n \times W_n(R)$, 
and $t \in R_+$,
 \Eq{*}{
   \M(x,\lambda)=\M(x,t \cdot \lambda),
 }
\item \emph{Reduction principle}: For all $n \in \N$ and for all $x \in I^n$, $\lambda,\mu \in W_n(R)$, 
\Eq{*}{
\M(x,\lambda+\mu)=\M(x\odot x,\lambda\odot\mu),
}
where $\odot$ is a \emph{shuffle operator} defined as
\Eq{*}{
(p_1,\dots,p_n)\odot (q_1,\dots,q_n):=(p_1,q_1,\dots,p_n,q_n).
}
\item \emph{Mean value property}: For all $n \in \N$, for all $(x,\lambda) \in I^n \times W_n(R)$
\Eq{*}{
\min(x_1,\dots,x_n) \le \M(x,\lambda)\le \max(x_1,\dots,x_n),
}
\item \emph{Elimination principle}: For all $n \in \N$, for all $(x,\lambda) \in I^n \times W_n(R)$ and for 
all $j\in\{1,\dots,n\}$ such that $\lambda_j =0$,
\Eq{*}{
\M(x,\lambda) = \M\big((x_i)_{i\in\{1,\dots,n\}\setminus\{j\}},(\lambda_i)_{i\in\{1,\dots,n\}\setminus\{j\}}\big),
}
i.e., entries with a zero weight can be omitted. 
\end{enumerate}
\end{defin}

Let us now introduce some important properties of weighted means. A weighted mean $\M$ is said to be \emph{symmetric}, if for all $n \in \N$, 
$(x,\lambda) \in I^n\times W_n(R)$, and $\sigma \in S_n$, 
\Eq{*}{
\M(x,\lambda) =\M(x\circ\sigma,\lambda\circ\sigma).
}
We will call a weighted mean $\M$ \emph{Jensen concave} if, for all $n \in \N$, $x,y \in I^n$ and $\lambda \in W_n(R)$,
\Eq{E:JF2}{
\M \Big( \frac{x+y}2 , \lambda \Big) \ge\frac12 \big( \M(x,\lambda)+\M(y,\lambda) \big).
}
If the above inequality holds with reversed inequality sign, then we speak about the \emph{Jensen convexity} of $\M$. Using that the mapping 
$x\mapsto\M(x,\lambda)$ is locally bounded, the Bernstein--Doetsch Theorem \cite{BerDoe15} implies that $\M$ is in fact concave or convex, respectively. 

A weighted mean $\M$ is said to be \emph{monotone} if, for all $n \in \N$ and $\lambda \in W_n(R)$, the mapping $x_i \mapsto \M(x,\lambda)$ is nondecreasing for all $i \in \{1,\dots,n\}$.

Finally, assuming that $I$ is a subinterval of $\R_+$, we call a weighted mean $\M$ \emph{homogeneous}, if for all $t>0$, $n\in\N$ and 
$(x,\lambda)\in \big(I\cap\frac1t I\big)^n\times W_n(R)$,
\Eq{*}{
  \M(tx,\lambda)=t\M(x,\lambda).
}

\section{\label{sec:homomean}Homogenizations of weighted means}

There are several possibilities to construct a homogeneous weighted mean out of a given one. To illustrate the main ideas let $\M \colon \bigcup_{n=1}^{\infty} I^n \times W_n(R) \to I$ be a given $R$-weighted mean.

\subsection{Homogeneous envelopes}

In what follows, we can construct two $R$-weighted means $\M_\$,\M^\$\colon \bigcup_{n=1}^{\infty} I^n \times W_n(R) \to I$ such that $\M_\$\leq\M\leq\M^\$$ in the following manner: For $n\in\N$, $x\in I^n$ and $\lambda\in W_n(R)$, define
\Eq{*}{
  \M_\$(x,\lambda):=\inf\big\{\tfrac1t\M(tx,\lambda)\mid tx\in I^n\big\}
  \quad\mbox{and}\quad
  \M^\$(x,\lambda):=\sup\big\{\tfrac1t\M(tx,\lambda)\mid tx\in I^n\big\}.
}
One can easily see that $\M_\$$ and $\M^\$$ are the largest and smallest $R$-weighted homogeneous means on $I$ such that $\M_\$\leq\M\leq\M^\$$ holds. Therefore, these means will be called the \emph{lower and upper homogeneous envelopes of $\M$}, respectively. It is obvious that $\M$ is homogeneous if and only $\M_\$=\M^\$$. It is also easy to check that if $\M$ is symmetric and monotone, then also $\M_\$$ and $\M^\$$ are symmetric and monotone, respectively. The gap between the lower and upper envelopes, however, could be too large, therefore it seems to be important to construct homogeneous means related to $\M$ that could provide us more precise information on $\M$.

\subsection{Local homogenizations}

For our next construction, we assume that $\inf I=0$. In this case, for every $x\in \R_+^n$, the inclusion $tx\in I^n$ holds for small positive $t$ which allows us to perform the limit operations as $t\to0$. For an $R$-weighted mean $\M$ on $I$ we now introduce the following homogenizations $\M_\#,\M^\# \colon \bigcup_{n=1}^{\infty} \R_+^n \times W_n(R) \to \R_+$ by
\Eq{*}{
  \M_\#(x,\lambda):=\liminf_{t\to 0}\tfrac1t\M(tx,\lambda) \qquad\mbox{and}\qquad \M^\#(x,\lambda):=\limsup_{t\to 0}\tfrac1t\M(tx,\lambda).
}
We call $\M_\#$ and $\M^\#$ the \emph{lower and upper (local) homogenization} of the $R$-weighted mean $\M$, respectively. It is obvious that $\M_\#$ and $\M^\#$ are homogeneous $R$-weighted means on $\R_+$, furthermore, we have the inequality $\M_\#\leq\M^\#$ on $\R_+$ and the inequalities $\M_\$\leq\M_\#\leq\M^\#\leq\M^\$$ on $I$. Furthermore, the mean $\M$ is homogeneous if and only if the equality $\M=\M_\#=\M^\#$ holds on the domain of $\M$. It is also easy to see that if $\M$ is symmetric and monotone, then also $\M_\#$ and $\M^\#$ are symmetric and monotone, respectively.

In the following result we establish a sufficient condition for the equality of $\M_\#$ and $\M^\#$. 

\begin{thm}
\label{thm:M*}
Let $I$ be a subinterval of $\R_+$ with $\inf I=0$ and $\M$ be a Jensen concave $R$-weighted mean on $I$. Then $\M_\#=\M^\#$ and these means are also Jensen concave. In addition, $\M\leq\M_\#=\M^\#$ on the domain of $\M$.
\end{thm}

\begin{proof} Let $n\in\N$ and $(x,\lambda) \in \R_+^n\times W_n(R)$ be fixed. Choose $\tau>0$ such that $\tau x\in I^n$. By the mean value property of $\M$, for all $t\in(0,\tau]$, we have 
\Eq{*}{
t\min(x_1,\dots,x_n) \le \M(tx,\lambda)\le t\max(x_1,\dots,x_n).
}
Therefore,
\Eq{*}{
  \lim_{t\to 0}\M(tx,\lambda)=0=:m(0).
}
On the other hand, by the concavity of the mean $\M$, the mapping $m(t):=\M(tx,\lambda)$ is concave, therefore the difference ratio 
\Eq{*}{
\frac{m(t)-m(0)}{t}=\frac{m(t)}{t}=\frac{1}{t}\M(tx,\lambda)
}
is a decreasing function of $t$ on the interval $(0,\tau]$. Thus, its right limit at zero exists, which implies the equality $\M_\#(x,\lambda)=\M^\#(x,\lambda)$. 

If $x\in I^n$, then $\tau=1$ can be chosen. Therefore, for all $t\in(0,1]$,
\Eq{*}{
\M(x,\lambda)=m(1)=\frac{m(1)-m(0)}{1}\leq\frac{m(t)-m(0)}{t}=\frac{m(t)}{t}=\frac{1}{t}\M(tx,\lambda).
}
As $t\to0$, the inequality $\M(x,\lambda)\leq \M_\#(x,\lambda)=\M^\#(x,\lambda)$ follows. 

To see that $\M_\#=\M^\#$ is Jensen concave, let $x,y\in \R_+^n$ and $\lambda\in W_n(R)$. Then there exists $\tau>0$ such that $\tau x,\tau y\in I^n$. Now, by the Jensen concavity of $\M(\cdot,\lambda)$, we get 
\Eq{*}{
  \frac{1}{t}\M\Big(t\frac{x+y}2,\lambda\Big)\geq \frac12\Big(\frac{1}{t}\M(tx,\lambda)+\frac{1}{t}\M(ty,\lambda)\Big) \qquad(t\in(0,\tau]).
}
Upon taking the limit $t\to0$, we obtain that
\Eq{*}{
  \M_\#\Big(\frac{x+y}2,\lambda\Big)
    \geq \frac12\big(\M_\#(x,\lambda)+\M_\#(y,\lambda)\big),
}
which proves the Jensen concavity of $\M_\#=\M^\#$.
\end{proof}

In the first part of the proof, we did not extensively use the Jensen concavity of the mean $\M$, we needed only the Jensen concavity of the map $t\mapsto\M(tx,\lambda)$ for all fixed pair $(x,\lambda)$. 

One can also get an analogous statement for Jensen convex means. In this case, the equality $\M_\#=\M^\#$ is also valid, these means are Jensen convex and $\M\geq\M_\#=\M^\#$ holds on the domain of $\M$.

\section{\label{sec:examples}Homogenizations of weighted quasiarithmetic means}

First we recall two basic classes of weighted means. 
For a parameter $p\in\R$, define the \emph{weighted power mean} $\P_p\colon \bigcup_{n=1}^{\infty} \R_+^n \times W_n(\R) \to \R_+$ by
\Eq{*}{
\P_p (x,\lambda):= 
\begin{cases} 
\left(\frac{\lambda_1x_1^p+\cdots+\lambda_nx_n^p}{\lambda_1+\cdots+\lambda_n} \right)^{1/p} &\quad \text{ if } p \ne 0, \\                                                          
\left(x_1^{\lambda_1}\cdots x_n^{\lambda_n} \right)^{1/(\lambda_1+\cdots+\lambda_n)} &\quad \text{ if } p = 0,\\
\min\big\{x_i\mid i\in\{1,\dots,n\},\,\lambda_i>0\big\} &\quad \text{ if }p = -\infty,\\
\max\big\{x_i\mid i\in\{1,\dots,n\},\,\lambda_i>0\big\} &\quad \text{ if }p = +\infty.
\end{cases}
}
It is well-known that, for every nonconstant vector $x\in I^n$ and positive weight vector $\lambda\in\R_+^n$, the mapping
\Eq{*}{
  [-\infty,\infty]\ni p\mapsto\P_p(x,\lambda)\in \big[\P_{-\infty}(x,\lambda),\P_{\infty}(x,\lambda)\big]
}
is an increasing bijection.

In a more general setting, we can define weighted quasiarithmetic means in the spirit of \cite{HarLitPol34}. Given an interval $I$ and a continuous strictly monotone function $f \colon I \to \R$, the \emph{weighted quasiarithmetic mean} 
$\QA{f} \colon \bigcup_{n=1}^{\infty} I^n \times W_n(\R) \to I$ is defined by
\Eq{QA}{
\QA{f}(x,\lambda):= f^{-1} \left( \frac{\lambda_1 f(x_1)+\cdots+\lambda_nf(x_n)}{\lambda_1+\cdots+\lambda_n} \right).
}

With the notations $\pi_p(x):=x^p$ if $p\ne 0$ and $\pi_0(x):=\ln x$, we can immediately see that the equality $\QA{\pi_p}=\P_p$ holds. It is well-known that the homogeneous quasiarithmetic means are exactly the power means (cf. \cite{HarLitPol34}, \cite[Theorem 4]{Pal00a}, \cite{Pas15a}).

In 1948, Mikusi\'nski proved that the comparability problem within this family  (under natural smoothness 
assumptions) is equivalent to the pointwise comparability of the mapping $f\mapsto \frac{f''}{f'}$ (negative of this 
operator is called the \emph{Arrow-Pratt index of absolute risk aversion}; cf.\ \cite{Arr65}, \cite{Pra64} ). More precisely, he proved

\begin{prop}\label{Mik}
Let $I \subset \R$ be an interval, $f,\,g \colon I \to \R$ be twice differentiable functions having nowhere vanishing  
first derivative. Then the following two conditions are equivalent
 \begin{itemize}
 \item[\upshape{(i)}] $\QA{f}(x,\lambda)\leq\QA{g}(x,\lambda)$ 
 for all $n\in\N$ and $(x,\lambda)\in I^n\times W_n(\R)$;
 \item[\upshape{(ii)}] $\frac{f''(x)}{f'(x)}\leq \frac{g''(x)}{g'(x)}$ for all $x\in I$.
 \end{itemize}
\end{prop}

In a special case $I\subseteq\R_+$ condition (ii) can be equivalently written as
\Eq{*}{
  \chi_f(x):=\frac{xf''(x)}{f'(x)}+1\leq\frac{xg''(x)}{g'(x)}+1=:\chi_g(x) \qquad(x\in I).
}
It is easy to verify that the equality $\chi_{\pi_p}\equiv p$ holds for 
all $p\in\R$. Therefore, in view of Proposition~\ref{Mik} and the definitions of the lower and upper homogeneous envelopes, we have
\Eq{*}{
\P_q=\QA{\pi_q} \le \big(\QA{f}\big)_\$ 
   \le \QA{f} \le \big(\QA{f}\big)^\$ \le \QA{\pi_p}=\P_p,
}
where $q:=\inf_I\chi_f$ and $p:=\sup_I\chi_f$, moreover 
these parameters are sharp. In other words, the operator $\chi_{(\cdot)}$ could be applied to embed 
quasiarithmetic means into the scale of power means (cf. \cite{Pas13}). 

This property can be used to obtain a homogenization of quasiarithmetic means.
\begin{thm}
 \label{thm:HomQA}
 Let $I$ be an interval with $\inf I=0$ and $f \colon I \to \R$ be a twice differentiable function with nowhere vanishing first derivative. Define $p,\,q \in [-\infty,+\infty]$ by 
  \Eq{*}{
  p:=\liminf_{t \to 0} \chi_f(t) \qquad\text{ and }\qquad q:=\limsup_{t \to 0} \chi_f(t).
  }
  Then $\P_p \le (\QA{f})_\#$ and $(\QA{f})^\# \le \P_q$. 
  
  In particular, if the limit $\lim_{t \to 0} \chi_f(t)$ exists then $(\QA{f})_\#=(\QA{f})^\#=\P_p=\P_q$.
\end{thm}
\begin{proof}
Let $n\in\N$ and $(x,\lambda) \in \R_+^n\times W_n(R)$ be fixed.
We will prove that $(\QA{f})^\#(x,\lambda) \le \P_q(x,\lambda)$.
If $q=+\infty$ then this inequality is trivially satisfied. From now on suppose that $q \in [-\infty,\infty)$.

Fix $q_0>q$ arbitrarily (if $q=-\infty$ then $q_0\in \R$) and choose $\delta>0$ such that 
\Eq{*}{
\chi_f(t) \le q_0=\chi_{\pi_{q_0}}(t) \qquad \text{ for all }t \in (0,\delta).
}
By Proposition~\ref{Mik}, this inequality implies that $\QA{f}\leq \QA{\pi_{q_0}}=\P_{q_0}$ holds on the interval $(0,\delta)$. In particular, if $\tau>0$ satisfies $\tau x\in (0,\delta)^n$, then we have
 \Eq{*}{
 \tfrac{1}t \QA{f}(tx,\lambda) \le \P_{q_0}(x,\lambda) \qquad \text{ for all }t \in (0,\tau).
 }
Upon taking the limsup as $t\to0$, it follows that $(\QA{f})^\#(x,\lambda) \le \P_{q_0}(x,\lambda)$. As $q_0>q$ was taken arbitrarily, we obtain $(\QA{f})^\#(x,\lambda) \le \P_{q}(x,\lambda)$, which completes the proof of the inequality $(\QA{f})^\# \le \P_{q}$. Similarly, we can also obtain the inequality $\P_{p}\leq (\QA{f})_\#$.

Finally, to prove the last part of this statement, notice that in this particular case we have $p=q$.
\end{proof}

In the next statement, we establish a result which is valid without smoothness assumptions.

\begin{prop}
Let $I$ be an interval with $\inf I=0$ and let $f\colon\to\R$ be a strictly monotone continuous function. Then the following assertions are equivalent:
\begin{enumerate}[(i)]
\item There exits $p\in\R$ such that $(\QA{f})^\#=(\QA{f})_\#=\P_p$;
\item There exits $p\in\R$ such that, for all $x,y,z\in\R_+$ with $z\neq y$,
\Eq{*}{
 \lim_{t \to 0} \frac{f(tx)-f(ty)}{f(tz)-f(ty)}=\frac{\pi_p(x)-\pi_p(y)}{\pi_p(z)-\pi_p(y)};
}
\item The function 
\Eq{*}{
   \varphi \colon \R_+ \ni x \mapsto \lim_{t \to 0} \frac{f(tx)-f(t)}{f(2t)-f(t)}
}
is well-defined (i.e., the limit exists for all $x>0$), is continuous and strictly monotone.
\end{enumerate}
\end{prop}

\begin{proof}
It is known \cite[Corollary~1]{Pal91} that if $\{f_t\mid t\in[0,1]\}$ is a family of continuous, strictly monotone functions on $I$, then
\Eq{*}{
\QA{f_t} \to \QA{f_0} \iff \lim_{t \to 0}\frac{f_t(x)-f_t(y)}{f_t(x)-f_t(z)} = \frac{f_0(x)-f_0(y)}{f_0(x)-f_0(z)} \text{ for all } x,\,y,\,z \in I \text{ with } x \ne z\,.
}
Defining
\Eq{*}{
  f_t(x):=f(tx) \qquad(t\in(0,1],\,x\in I), \qquad f_0(x)=\pi_p(x) \qquad(x\in I)
}
and applying the above statement, we can easily see that the assertions (i) and (ii) are equivalent. 

The substitutions $y=1$, $z=2$ in assertion (ii) imply that (iii) is valid.
To prove the implication (iii)$\Rightarrow$(i), first choose $\tau>0$ such that $2\tau\in I$, then let
\Eq{*}{
 \varphi_t(x):=\frac{f(tx)-f(t)}{f(2t)-f(t)}=a_t f_t(x)+b_t \qquad(t\in(0,\tau],\,x\in I).
}
Applying the well-known criteria for the equality of quasiarithmetic means, we obtain 
\Eq{*}{
\QA{\varphi_t}(x,\lambda)=\QA{f_t}(x,\lambda)
=\tfrac1t \QA{f}(t x,\lambda)
\text{ \ for all }n \in \N\text{ and }(x,\lambda) \in I^n \times W_n(\R).
}
Furthermore we know that $\varphi_t \to \varphi$ as $t \to 0$. Whence 
\Eq{*}{
\lim_{t \to 0} \tfrac1t \QA{f}(t x,\lambda)=\lim_{t \to 0} \QA{\varphi_t}(x,\lambda)=\QA{\varphi}(x,\lambda) \text{ \ for all }n \in \N\text{ and }(x,\lambda) \in I^n \times W_n(\R).
}
Thus $(\QA{f})_\#=(\QA{f})^\#=\QA{\varphi}$. However, obviously, homogenization of a mean is a homogeneous mean; therefore $\QA{\varphi}$ is a homogeneous quasiarithmetic mean, that is, $\QA{\varphi}=\P_p$ for some $p \in \R$ (cf.\ \cite{HarLitPol34}), which proves (i). 
 \end{proof}

\section{\label{sec:semidev}Semideviations and semideviation means}

In what follows, we recall the notion of a semideviation and the weighted semideviation means generated by them (cf. \cite{Pal89b}). 

Let us stress that in the mentioned paper \cite{Pal89b} this definition was restricted to two means ($\LLD$ and $\UUD$ in the definition below; it was called lower an upper semideviation, respectively). Furthermore it was presented in nonweighted setting only. As semideviations are not continuous in their weights, extensions theorem \cite[Theorems 2.2, 2.3]{PalPas18b} allows us to extend it to a $\Q$-weighted mean only. Its extension to $\R$-weighted mean was introduced as it should remain a natural generalization of quasideviation means.

In the present paper we are going to introduce two more semideviation means ($\LD$ and $\UD$) and extend this property to $\R$-weighted means.

\begin{defin}[Semideviations and weighted semideviation means]
A real valued function $E\colon I\times I\to\R$ is said to be a \emph{semideviation} if, for all distinct elements $x,y\in I$, the sign of $E(x,y)$ coincides with that of $x-y$.
For $n\in\N$, $(x,\lambda)\in I^n\times W_n(\R)$, consider the function $e:I\to\R$ defined by 
\Eq{e}{
  e(y):=\lambda_1 E(x_1,y)+\cdots+\lambda_n E(x_n,y).
}
In terms of the points where the function $e$ changes its sign, we can introduce four means as follows:
\Eq{*}{
\LLD_E(x,\lambda)&:= \inf\big\{y \in I \mid e(y) \le 0 \big\}, \quad
&\UUD_E(x,\lambda)&:= \sup\big\{y \in I \mid e(y) \ge 0 \big\},\\
\LD_E(x,\lambda)&:= \inf\big\{y \in I \mid e(y) < 0 \big\}, \quad
&\UD_E(x,\lambda)&:= \sup\big\{y \in I \mid e(y) > 0 \big\}.
}
We can easily see that
\Eq{*}{
\LLD_E(x,\lambda)&= \sup\big\{y \in I \colon e|_{(-\infty,y)\cap I}>0\big\}, \quad
&\UUD_E(x,\lambda)&= \inf\big\{y \in I \colon e|_{(y,\infty)\cap I}<0\big\},\\
\LD_E(x,\lambda)&= \sup\big\{y \in I \colon e|_{(-\infty,y)\cap I}\geq0\big\}, \quad
&\UD_E(x,\lambda)&= \inf\big\{y \in I \colon e|_{(y,\infty)\cap I}\leq0\big\}.
}
The two functions $\LLD_E,\UUD_E \colon \bigcup_{n=1}^{\infty} I^n \times W_n(\R) \to I$ will be called the \emph{lower semideviation} and \emph{upper semideviation mean} generated by $E$, respectively. 
\end{defin}

The following lemma is obvious but very useful.

\begin{lem} Let $E\colon I\times I\to\R$ be a semideviation and $n\in\N$, $x=(x_1\dots,x_n)\in I^n$, $\lambda=(\lambda_1\dots,\lambda_n)\in W_n(\R)$. Define the function $e\colon I\to\R$ by \eq{e}.
\begin{enumerate}[(i)]
\item If the function $e$ is nonincreasing, then 
\Eq{*}{
  \LLD_E(x,\lambda)=\UD_E(x,\lambda) 
  \qquad\mbox{and}\qquad \LD_E(x,\lambda)=\UUD_E(x,\lambda);
}
\item If the function $e$ is strictly decreasing, then 
\Eq{*}{
  \LLD_E(x,\lambda)=\LD_E(x,\lambda)=\UD_E(x,\lambda)=\UUD_E(x,\lambda).
}
\item If the function $e$ is continuous, then $\LLD_E(x,\lambda)$,\, $\UUD_E(x,\lambda)$,\, $\LD_E(x,\lambda)$, and $\UD_E(x,\lambda)$ are zeroes of the function $e$.
\end{enumerate}
\end{lem}

The next statement provides us the most basic inequalities among the four semideviation means.

\begin{prop}\label{prop:SDM}
Let $E\colon I\times I\to\R$ be a semideviation. Then, the means $\LLD_E$,\, $\UUD_E$,\, $\LD_E$, and $\UD_E$ are symmetric $\R$-weighted means on $I$. In addition, for all $n\in\N$, $(x,\lambda)\in I^n\times W_n(\R)$,
\Eq{SDM}{
  \min(x_1,\dots,x_n)\leq \LLD_E(x,\lambda)\leq \UUD_E(x,\lambda)\leq \max(x_1,\dots,x_n)
}
and
\Eq{*}{
\LD_E(x,\lambda), \UD_E(x,\lambda) \in \big[\LLD_E(x,\lambda),\,\UUD_E(x,\lambda)\big].
}
\end{prop}

\begin{proof}
First observe that the reduction and the elimination principles as well as the symmetry are trivial as these transformations do not affect the function $e$. To prove the nullhomogeneity in the weights, observe that formula for $e$ is  linear with repsect to $\lambda$. Thus, multiplying $\lambda$ by a positive constant, preserves the sign of $e$ at every point.  

Let $n\in\N$, $(x,\lambda)\in I^n\times W_n(\R)$. 
By the semideviation property of $E$, for $y<\min(x_1,\dots,x_n)$, we have $e(y)>0$, therefore, the inequality 
$\min(x_1,\dots,x_n)\leq \LLD_E(x,\lambda)$ is always valid. On the other hand, if $\max(x_1,\dots,x_n)<y$, then $e(y)<0$. Therefore, 
$\UUD_E(x,\lambda)\leq \max(x_1,\dots,x_n)$.

If the second inequality in \eq{SDM} were not true, i.e., if $\UUD_E(x,\lambda)<\LLD_E(x,\lambda)$. Then, for any element $y$ between these values, we would have that $e(y)<0<e(y)$, which would be an obvious contradiction.

To prove the last assertion, observe that $e(z)>0$ for all $z \le \LLD_E(x,\lambda)$. Then, by the definition $\LD_E(x,\lambda) \ge \LLD_E(x,\lambda)$ and $\UD_E(x,\lambda) \ge \LLD_E(x,\lambda)$ (as $e(z) \le 0$ for no $z \le \LLD_E(x,\lambda)$). The proof of the two remaining inequalities is completely analogous.
\end{proof}

\begin{defin}[Quasideviations and weighted quasideviation means]
A semideviation $E\colon I\times I\to\R$ is said to be \emph{quasideviation} if, 
\begin{enumerate}[(a)]
 \item for all $x\in I$, the map $y\mapsto E(x,y)$ is continuous and,
 \item for all $x<y$ in $I$, the mapping $(x,y)\ni t\mapsto \frac{E(y,t)}{E(x,t)}$ is strictly increasing.
\end{enumerate}
By the results of the paper \cite{Pal82a}, it follows that lower semideviation and upper semideviation means generated by a quasideviation $E$ coincide. Moreover, for $n\in\N$, $(x,\lambda)\in I^n\times W_n(\R)$, the value $y=\D_E(x,\lambda):=\LLD_E(x,\lambda)=\UUD_E(x,\lambda)$ is the unique solution of the equation $e(y)=0$, where the function $e$ is defined by \eq{e}. The mean $\D_E$ defined this way, will be called the \emph{quasideviation mean} generated by $E$.
\end{defin}

\begin{exa}Now we provide some examples for the notions introduced above.
\begin{enumerate}[(1)]
 \item The function $S(x,y):=\sign(x-y)$ is a semideviation on $\R$. The lower and upper semideviation means generated by this semideviation define weighted lower and upper \emph{median} on $\R$.
 \item The function $A(x,y):=x-y$ is a quasideviation on $\R$, which generates the weighted arithmetic mean $\A=\P_1$.
 \item For a strictly increasing function $f\colon I\to\R$, let $f^{-1}\colon J\to I$ be the uniquely determined nondecreasing extension of the inverse function of $f$ to the smallest interval $J$ containing $f(I)$. Then the function $A_f$ defined by $A_f(x,y):=f(x)-f(y)$ is a semideviation on $I$. Furthermore, it can be proved that the lower and upper $A_f$-semideviation means coincide and, for $n\in\N$, $(x,\lambda)\in I^n\times W_n(\R)$, these means are given by formula \eq{QA}. Due to the continuity of $f^{-1}$, it follows that this weighted mean is continuous in the weights, monotone, but it may not be continuous in its entries. 
 \item Let $f\colon\R_+\to\R$ be any function such that $\sign(f(x))=\sign(x-1)$ for all $0<x\neq1$. Then $E_f(x,y):=f\big(\frac{x}{y}\big)$ is a semideviation on $\R_+$, furthermore, it is easy to see that the four weighted semideviation means 
 \Eq{Ef}{
   \LLE_f:=\LLD_{E_f}, \qquad \LE_f:=\LD_{E_f} \qquad
   \mbox{and}\qquad \UE_f:=\UD_{E_f}, \qquad \UUE_f:=\UUD_{E_f}
}
are homogeneous. If $f$ is continuous and strictly increasing, then it can be proved that $E_f$ is a quasideviation on $\R_+$, the above four weighted semideviation means coincide and are equal to the quasideviation mean $\E_f$ which is monotone and continuous.
\end{enumerate}
\end{exa} 

\section{Means generated by normalizable semideviations}

In this section we prove that if $E$ satisfies some additional assumptions, we can establish several counterparts of the well-known results considering deviation and quasideviation means. This notion is motivated by already introduced definition of regular quasideviation in \cite{Pal88a}. 

 We say that a semideviation $E\colon I\times I\to\R$ is \emph{normalizable} if
 \begin{enumerate}[(a)]
  \item for all $x\in I$, the function $I\ni y\mapsto E(x,y)$ is continuous,
  \item for all $x\in I$, the function $I\ni y\mapsto E(x,y)$ is differentiable at $x$, and
  \item the mapping $x\mapsto\partial_2E(x,x)$ is strictly negative and continuous on $I$.
 \end{enumerate}
For a normalizable semideviation $E$, we define its \emph{normalization} $E^*\colon I\times I\to\R$ of $E$ by
\Eq{*}{
  E^*(x,y):=\frac{E(x,y)}{-\partial_2E(y,y)} \qquad(x,y\in I).
}
It is obvious that the corresponding semideviation means generated by $E$ and $E^*$ are identical. The next statement shows that, for a normalized  semideviation $E$, the partial derivative $\partial_2E$ is identically equal to $-1$ on the diagonal of $I\times I$.

\begin{lem}
Let $E\colon I\times I\to\R$ be a normalizable semideviation. Then $E^*$ is also a normalizable semideviation and $\partial_2E^*(x,x)=-1$ for all $x\in I$, and hence $(E^*)^*=E^*$ holds, too.
\end{lem}

\begin{proof} Let $x\in I$ be fixed. The function $E(x,\cdot)$ is differentiable at $x$, hence it is also continuous here, which together with the semideviation property implies $E(x,x)=0$. Then
\Eq{*}{
  \partial_2E^*(x,x)
  &=\lim_{y\to x}\frac{E^*(x,y)-E^*(x,x)}{y-x}
  =\lim_{y\to x}\frac{E^*(x,y)}{y-x}
  =\lim_{y\to x}\frac{E(x,y)-E(x,x)}{-\partial_2E(y,y)(y-x)}\\
  &=\lim_{y\to x}\frac{E(x,y)-E(x,x)}{y-x}\lim_{y\to x}\frac{1}{-\partial_2E(y,y)} 
  =\partial_2E(x,x)\cdot\frac{1}{-\partial_2E(x,x)}=-1.
}
Now, the identity $(E^*)^*=E^*$ follows immediately.
\end{proof}

The comparison of weighted lower and upper semideviation means can be characterized by the following theorem which is an extension of the analogous results obtained by Daróczy \cite{Dar71b,Dar72b} and by Páles \cite{Pal88a}.
The key to prove this theorem is the following auxiliary result.

\begin{lem}\label{Lim}
Let $E\colon I\times I\to\R$ be a normalizable semideviation. Then, for all $x,y\in I$,
\Eq{Lim1}{
  \lim_{n\to\infty}n\Big(\LLD_E\big((x,y),(1,n)\big)-y\Big)
  =\lim_{n\to\infty}n\Big(\UUD_E\big((x,y),(1,n)\big)-y\Big)
  =E^*(x,y).
}
\end{lem}

\begin{proof} If $x=y$, then there is nothing to prove. Assume that $x<y$ and set
\Eq{*}{
  u_n:=\LLD_E\big((x,y),(1,n)\big),\qquad v_n:=\UUD_E\big((x,y),(1,n)\big)\qquad(n\in\N).
}
First we are going to show that both sequences $(u_n)$ and $(v_n)$ converge to $y$. Let $0<\varepsilon<y-x$ be arbitrary. It is enough to show that, for large $n$, the inequality $y-\varepsilon\leq u_n$ holds, because we also have $u_n\leq v_n\leq y$.

The function 
\Eq{*}{
  t\mapsto -\frac{E(x,t)}{E(y,t)}
}
is continuous on $(-\infty,y-\varepsilon]\cap I$, negative for $t<x$ and nonnegative for $t\in[x,y-\varepsilon]$. Thus, due to its continuity, 
\Eq{*}{
  K_\varepsilon:=\sup_{t\in(-\infty,y-\varepsilon]\cap I}-\frac{E(x,t)}{E(y,t)}
}
is finite. Therefore, for $n>K_\varepsilon$ and $t\in(-\infty,y-\varepsilon]$, we have
\Eq{en}{
 e_n(t):=E(x,t)+nE(y,t)>0.
}
This implies that, for $n>K_\varepsilon$, 
\Eq{*}{
   y-\varepsilon\leq u_n,
}
and hence we get that $(u_n)$ (together with $(v_n)$) converges to $y$. (The proof in the case $y<x$ is completely analogous.) 

Applying the continuity of the function $e_n$ defined in \eq{en}, it follows that $u_n$ and $v_n$ are zeroes of the function $e_n$, that is,
\Eq{en+}{
  E(x,u_n)+nE(y,u_n)=0 \qquad\mbox{and}\qquad E(x,v_n)+nE(y,v_n)=0.
}
It is obvious that $u_n\neq y\neq v_n$. 
Then the first equality in \eq{en+} yields
\Eq{*}{
  n(u_n-y)=-E(x,u_n)\frac{u_n-y}{E(y,u_n)-E(y,y)}.
}
Using the continuity of $E$ in the second variable, the regularity assumption and that $(u_n)$ tends to $y$, it follows that the right hand side of the above equality tends to $E^*(x,y)$ as $n\to\infty$. Therefore, $n(u_n-y)$ converges to the same limit. Analogously, starting from the second equality in \eq{en+}, we can obtain that $n(v_n-y)$ also converges to $E^*(x,y)$. Thus the proof is complete.
\end{proof}

\begin{thm}\label{com}
Let $E,F\colon I\times I\to\R$ be normalizable semideviations. If $E^*\leq F^*$, then
\Eq{EF}{
  \LLD_E\leq \LLD_F, \qquad\LD_E\leq \LD_F \qquad\mbox{and}\qquad 
  \UD_E\leq \UD_F,\qquad \UUD_E\leq \UUD_F.
}
Conversely, if
\Eq{EF+}{
  \LLD_E\leq\UUD_F,
}
then $E^*\leq F^*$. Therefore, the inequality $E^*\leq F^*$ is also necessary for any the of the inequalities in \eq{EF}.
\end{thm}

\begin{proof} 
Let $n\in\N$, $x\in I^n$ and $\lambda\in W_n(\R)$. Then, for all $y<\LLD_E(x,\lambda)$, we have that
\Eq{*}{
  \lambda_1 E^*(x_1,y)+\cdots+\lambda_n E^*(x_n,y)>0.
}
Therefore, the inequalities $E^*(x_i,y)\leq F^*(x_i,y)$ imply that 
\Eq{*}{
  \lambda_1 F^*(x_1,y)+\cdots+\lambda_n F^*(x_n,y)>0,
}
for all $y<\LLD_E(x,\lambda)$. Therefore, the inequality 
$\LLD_E(x,\lambda)\leq\LLD_F(x,\lambda)$ follows. The proof of the other inequality in \eq{EF} is completely analogous.

Assume that $\LLD_E\leq \UUD_F$ holds on $\bigcup_{n=1}^{\infty} I^n \times W_n(\R)$. Let $x,y\in I$ be fixed. Then, for all $n\in\N$, we have
\Eq{*}{
  n\Big(\LLD_E\big((x,y),(1,n)\big)-y\Big)
  \leq n\Big(\UUD_F\big((x,y),(1,n)\big)-y\Big).
}
According to Lemma \ref{Lim}, the left and the right hand sides of this inequality converge to $E^*(x,y)$ and $F^*(x,y)$, respectively. Hence $E^*\leq F^*$ holds on $I\times I$.
\end{proof}

A surprising consequence of the above comparison theorem is that, under its assumptions, the inequality in \eq{EF+} implies all the inequalities in \eq{EF}. On the other hand, due to Proposition \ref{prop:SDM}, each of the inequalities in \eq{EF} implies \eq{EF+}, hence all these inequalities are pairwise equivalent.

In what follows, we characterize the Jensen concavity (convexity) of semideviation means.

\begin{thm}\label{JC}
Assume that $E\colon I\times I\to\R$ is a normalizable semideviation. Then the following statements are equivalent:
\begin{enumerate}[(i)]
 \item Any of the means $\LLD_E$, $\LD_E$, $\UD_E$, and $\UUD_E$ is Jensen concave; 
 \item For all $n\in\N$, $x,y\in I^n$ and $\lambda\in W_n(\R)$,
 \Eq{*}{
   \UUD_E\Big(\frac{x+y}{2},\lambda\Big)
   \geq\frac{\LLD_E(x,\lambda)+\LLD_E(y,\lambda)}{2};
 }
 \item $E^*$ is Jensen concave on $I\times I$;
 \item $E^*$ is a quasideviation (hence so is $E$) and $\D_E$ is Jensen concave;
 \item All of the means $\LLD_E$, $\LD_E$, $\UD_E$, and $\UUD_E$ are Jensen concave.
\end{enumerate}
\end{thm}

\begin{proof} (i)$\Rightarrow$(ii). If $\M$ is any of the means $\LLD_E$, $\LD_E$, $\UD_E$, and $\UUD_E$, then by Proposition \ref{prop:SDM}, we have that $\LLD_E\leq \M\leq\UUD_E$. Thus, the Jensen-concavity of $\M$ implies
\Eq{*}{
   \UUD_E\Big(\frac{x+y}{2},\lambda\Big)\geq \M\Big(\frac{x+y}{2},\lambda\Big)
   \geq\frac{\M(x,\lambda)+\M(y,\lambda)}{2}
   \geq\frac{\LLD_E(x,\lambda)+\LLD_E(y,\lambda)}{2},
}
which was to be shown.

(ii)$\Rightarrow$(iii). To prove the Jensen concavity of $E^*$, let $x,y,u,v\in I$. Then, applying (ii) for the vectors $(x,u)$ and $(y,v)$ with the weight vector $(1,n)$, we get
\Eq{*}{
  n\Big(\UUD_E\Big(\Big(\frac{x+y}{2},\frac{u+v}{2}\Big),&(1,n)\Big)-\frac{u+v}{2}\Big)\\
  & \geq\frac{n(\LLD_E((x,u),(1,n))-u)+n(\LLD_E((y,v),(1,n))-v)}{2}.
}
Upon taking the limit $n\to\infty$ and applying Lemma \ref{Lim}, it follows that
\Eq{*}{
  E^*\Big(\frac{x+y}{2},\frac{u+v}{2}\Big)\geq \frac{E^*(x,u)+E^*(y,v)}{2},
}
which shows the Jensen concavity of $E^*$. On the other hand, $E^*$ is bounded from below (by zero) on the open set $\{(x,y)\in I^2\colon x>y\}$, therefore, the Bernstein--Doetsch theorem (cf.\ \cite{BerDoe15}) implies that it is also concave, consequently, it is also continuous.

(iii)$\Rightarrow$(iv). To verify that $E^*$ is a quasideviation, let $x<y$ be arbitrary from $I$. We need to show that the negative function
\Eq{*}{
  r(t):=\frac{E^*(y,t)}{E^*(x,t)}
}
is strictly increasing over the interval $(x,y)$. To the contrary, assume that there exist $t_1,t_2\in(x,y)$ with $t_1<t_2$, such that $r(t_1)\geq r(t_2)$. If $r(t_1)>r(t_2)$ then, using the Bolzano Mean Value Theorem for continuous functions and that $r(y)=0$, we can find an element $t_2'\in(t_2,y)$ such that $r(t_1)=r(t_2')$. Therefore, we may assume that already $r(t_1)=r(t_2)$ was valid. Define $\lambda:=\frac{-r(t_1)}{1-r(t_1)}=\frac{-r(t_2)}{1-r(t_2)}\in(0,1)$. By the definition of $r(t)$, it follows that
\Eq{*}{
  \lambda E^*(x,t_1)+(1-\lambda) E^*(y,t_1)=0=\lambda E^*(x,t_2)+(1-\lambda). E^*(y,t_2)
}
This means that the function
\Eq{*}{
  e(t):=\lambda E^*(x,t)+(1-\lambda) E^*(y,t) \qquad(t\in I)
}
has two distinct zeroes in $(x,y)$. On the other hand, by the concavity of $E^*$, the function $e$ is concave. Thus, $e(t)\leq 0$ for $t\in I\setminus(t_1,t_2)$, which contradicts the inequality $e(t)>0$ which is valid for $t<x$. The identity
\Eq{*}{
  r(t):=\frac{E^*(y,t)}{E^*(x,t)}=\frac{E(y,t)}{E(x,t)}
}
shows that $E$ is also a quasideviation. Due to this property, the four means 
$\LLD_E$, $\LD_E$, $\UD_E$, and $\UUD_E$ coincide and are equal to the deviation mean $\D_E=\D_{E^*}$. 

To complete the proof of implication (iii)$\Rightarrow$(iv) it remains to show that $\D_E$ is Jensen concave. Let $n\in\N$, $x=(x_1,\dots,x_n),y=(y_1,\dots,y_n)\in I^n$ and $\lambda=(\lambda_1,\dots,\lambda_n)\in W_n(\R)$. Denote 
$u:=\D_E(x,\lambda)$, $v:=\D_E(y,\lambda)$. Then, we have
\Eq{*}{
  \sum_{i=1}^n\lambda_i E^*(x_i,u)=0=\sum_{i=1}^n\lambda_i E^*(y_i,v).
}
On the other hand, using these equalities, the Jensen concavity of the function $E^*$, it follows that
\Eq{*}{
  \sum_{i=1}^n\lambda_i E^*\Big(\frac{x_i+y_i}{2},\frac{u+v}{2}\Big)
  \geq \sum_{i=1}^n\lambda_i\frac{E^*(x_i,u)+E^*(y_i,v)}{2}=0.
}
According to the definition of quasideviation means, this inequality implies that
\Eq{*}{
  \D_E\Big(\frac{x+y}{2},\lambda\Big)\geq \frac{u+v}{2}
  =\frac{\D_E(x,\lambda)+\D_E(y,\lambda)}{2}.
}

(iv)$\Rightarrow$(v). If $E$ is a quasideviation then all the four means $\LLD_E$, $\LD_E$, $\UD_E$, and $\UUD_E$ are identical and are equal to the Jensen concave deviation mean $\D_E=\D_{E^*}$.

The implication (v)$\Rightarrow$(i) is obvious.
\end{proof}

\section{Homogenizations of semideviation means}

In what follows, we provide lower and upper estimates for the lower and upper homogenization of semideviation means.
Let $I$ be a nonempty open interval with $\inf I=0$. Given a normalizable semideviation $E\colon I\times I\to\R$, let us define the \emph{lower and upper homogenization} 
$\underline{h}_E, \overline{h}_E$ of $E$ by
\Eq{*}{
  \underline{h}_E(x):=\liminf_{t\to0}\frac{E^*(xt,t)}{t} \qquad\mbox{and}\qquad 
  \overline{h}_E(x):=\limsup_{t\to0}\frac{E^*(xt,t)}{t} \qquad(x\in\R_+).
}
With these notations, we have

\begin{thm}\label{TEI}
Let $E\colon I\times I\to\R$ be a normalizable semideviation.
Assume that, for $x\in\R_+$, we have $-\infty<\underline{h}_E(x)\leq\overline{h}_E(x)<+\infty$ and
\Eq{*}{
  \sign(\underline{h}_E(x))=\sign(\overline{h}_E(x))=\sign(x-1).
}
Then
\Eq{EI}{
  \LLE_{\underline{h}_E}\leq \big(\UD_E\big)_\# \qquad\mbox{and}\qquad 
  \big(\LD_E\big)^\#\leq \UUE_{\overline{h}_E}.
}
\end{thm}

\begin{proof}
Let $n\in\N$, $x\in \R_+^n$ and $\lambda\in W_n(\R)$. To the contrary of the first inequality in \eq{EI}, assume that there exists $y\in I$ such that
\Eq{*}{
\liminf_{t\to0}\frac{\UD_E(tx,\lambda)}{t}
  =\big(\UD_E\big)_\#(x,\lambda)
   <y<\LLE_{\underline{h}_E}(x,\lambda).
}
Then there exists a positive null sequence $(t_m)$ such that, for all $m\in\N$, 
\Eq{*}{
  \UD_E(t_mx,\lambda)<t_my.
}
It follows from the definition of the upper semideviation mean $\UD_E$ that, for all $m\in\N$, 
\Eq{*}{
  \sum_{i=1}^n\lambda_iE(t_mx_i,t_my)\le0,
}
therefore
\Eq{*}{
  \sum_{i=1}^n\lambda_iE^*(t_mx_i,t_my)\le0.
}
Define $y_m:=t_my$ for $m\in\N$. Then
\Eq{*}{
  \sum_{i=1}^n\lambda_i\frac{E^*(y_m\frac{x_i}y,y_m)}{y_m}\le0.
}
Upon taking the liminf as $m\to\infty$ and using the superadditivity of the liminf operation, it follows that
\Eq{*}{
  \sum_{i=1}^n\lambda_i\liminf_{m\to\infty}\frac{E^*(y_m\frac{x_i}y,y_m)}{y_m}\leq0.
}
Therefore, 
\Eq{*}{
  \sum_{i=1}^n\lambda_i \underline{h}_E\Big(\frac{x_i}y\Big)
  =\sum_{i=1}^n\lambda_i\liminf_{t\to0}\frac{E^*(t\frac{x_i}y,t)}{t}\leq0,
}
which implies that $\LLE_{\underline{h}_E}(x,\lambda)\leq y$, contradicting the choice of $y$.
\end{proof}

%%%%%%%%%%%%%%%%%
\begin{exa}
 Let $E \colon \R_+ \times \R_+ \to \R$ be a deviation given by $E(x,y)=\cosh(x)-\cosh(y)$. Then $\D_E$ is a quasiarithmetic mean and 
 \Eq{*}{
 E^*(x,y)=\frac{\cosh(x)-\cosh(y)}{\sinh(y)} \qquad (x,y \in \R_+ ).
 }
As $E$ is a deviation we know that $\UD_E=\LD_E=\D_E$. Furthermore 
 \Eq{*}{
% \underline{h}_E(x)=
\lim_{t \to 0} \frac{E^*(tx,t)}{t}
 =\lim_{t \to 
 0} \frac{\cosh(tx)-\cosh(t)}{t\sinh(t)}
 &=\lim_{t \to 0} \frac{1+\tfrac12 (tx)^2-1-\tfrac12 t^2+\Err(t^4)}{t^2+\Err(t^4)}\\
&=\lim_{t \to 0}\frac{\tfrac12t^2(x^2-1+\Err(t^2))}{t^2(1+\Err(t^2))}=\frac{x^2-1}2.
 }
As this limit exists for all $x$ we obtain $\overline{h}_E(x)=\underline{h}_E(x)=\tfrac12(x^2-1)=:h_E(x)$. However $h_E$ is continuous and strictly increasing, thus the mapping $(x,y) \mapsto h_E(x/y)$ is a quasideviation on $\R_+$. Consequently $\LLE_{h_E}=\UUE_{h_E}=\E_{h_E}$ is a homogeneous quasideviation mean. 

Furthermore, for all $n\in \N$ and a pair $(x,\lambda) \in I^n \times W_n(\R)$, we get that $y:=\E_{h_E}(x,\lambda)$ is a unique positive solution of the equation
\Eq{*}{
\sum_{i=1}^n \dfrac{\lambda_i}2 \Big(\Big(\dfrac {x_i}y\Big)^2-1\Big)=0,
}
which simplifies to 
\Eq{*}{
\E_{h_E}(x,\lambda)=y=\sqrt{\frac{\lambda_1x_1^2+\dots+\lambda_nx_n^2}{\lambda_1+\dots+\lambda_n}}=\P_2(x,\lambda).
}
So $\E_{h_E}=\P_2$. Whence, by \eq{EI}, as lower and upper semideviation means of $E$ coincide, we get
\Eq{*}{
\P_2=\E_{h_E}=\LLE_{h_E}\le (\UD_E)_\#=(\D_E)_\#\le (\D_E)^\#=(\LD_E)^\#\le\UUE_{h_E}=\E_{h_E}=\P_2.
}
Therefore all means, appearing in the line above are equal to each other. In particular, the lower and upper homogenization of the quasiarithmetic mean  $\D_E$ equals the power mean $\P_2$ (which, of course, could also be computed directly).
\end{exa}

Our next result shows that the statement of Theorem \ref{TEI} dramatically simplifies if $E$ is a normalizable quasideviation such that $\D_E$ is Jensen concave.

\begin{thm}\label{CEI} Let $E\colon I\times I\to\R$ be a normalizable semideviation such that $E^*$ is concave. Assume that, for all $x\in\R_+$, $\lim_{t\to0}E^*(xt,t)=0$. Then, for all $x\in\R_+$, we have $h_E(x):=\underline{h}_E(x)=\overline{h}_E(x)\in\R$, 
\Eq{sign}{
  \sign(h_E(x))=\sign(x-1),
}
and the function $h_E\colon\R_+\to\R$ so defined is concave and nondecreasing on $\R_+$, and is strictly increasing on $(0,1)$. Furthermore, $E$ is a quasideviation, the weighted mean $\D_E$ is Jensen concave, monotone, and
\Eq{EI+}{
  \E_{h_E}=\big(\D_E\big)_\# =\big(\D_E\big)^\#.
}
\end{thm}

\begin{proof} Let $x\in\R_+$ be fixed. Consider the function
\Eq{*}{
   f(t):=\begin{cases}
         E^*(xt,t) &\mbox{ if } t>0,\,t\in I/x,\\
         0&\mbox{ if } t=0.
         \end{cases}
}
Then, from the concavity of $E^*$, it follows that $f$ is concave on $I/x$ hence it is continuous on $I/x$. It is also continuous at $0$ by our assumption, thus $f$ is concave on $(I/x)\cup\{0\}$. Due to well-known properties of concave functions, this implies that the difference ratio 
\Eq{*}{
  t\mapsto \frac{f(t)-f(0)}{t-0}=\frac{E^*(xt,t)}{t}
}
is nonincreasing on $I/x$. Hence its limit as $t\to0$, exists and is an extended real number in $\R\cup\{+\infty\}$. So far we have proved that $\underline{h}_E(x)=\overline{h}_E(x)=:h_E(x)$. If $x\in(1,\infty)$, then $h_E(x)\geq \frac{E^*(tx,t)}{t}>0$ for $t\in I/x$, therefore $h_E(x)>0$. We also have that $h_E(1)=0$ and $h_E(x)\leq0$ for $x\in(0,1)$.

On the other hand, for $x,y\in\R_+$ and $\lambda\in(0,1)$, applying the concavity of $E^*$, we get
\Eq{*}{
  h_E(\lambda x+(1-\lambda)y)
  &=\lim_{t\to0}\frac{E^*((\lambda x+(1-\lambda)y)t,t)}{t}\\
  &\geq\lim_{t\to0}\frac{\lambda E^*(xt,t)+(1-\lambda)E^*(yt,t)}{t}
  =\lambda h_E(x)+(1-\lambda) h_E(y).
}
This shows that $h_E$ is concave on $\R_+$. We show that it cannot take infinite values. Let $y\in\R_+$ be arbitrary, let $x=1/2$ and choose $\lambda\in(0,1)$ so that $\lambda x+(1-\lambda)y<1$. Then the above inequality shows that $h_E(y)<\infty$. 

If, for some $x\in(0,1)$, we had $h_E(x)=0$, then, by the concavity, $h_E$ would be nonpositive for $y\in\R_+\setminus(x,1)$. This contradiction completes the proof of the sign property of $h_E$.

Now we prove that $h_E$ is nondecreasing on $\R_+$. If this were not true, then there exist $0<x<y$ such that $h_E(x)>h_E(y)$. Let $a:\R_+\to\R$ denote the continuous affine function which interpolates $h_E$ at the points $x$ and $y$. The inequality $h_E(x)>h_E(y)$ implies that $\lim_{t\to\infty} a(t)=-\infty$. On the other hand, by the concavity of $h_E$, we have that $h_E(t)\leq a(t)$ for all $t\geq y$. Therefore, $\lim_{t\to\infty} h_E(t)=-\infty$, which contradicts the sign property of $h_E$ stated in \eq{sign}. In a similar manner, if $h_E$ were not strictly increasing in $(0,1)$, then there exist $0<x<y<1$ such that $h_E(x)=h_E(y)<0$. Now the interpolating affine function is a constant function with a negative value. Hence $\lim_{t\to\infty} h_E(t)<0$, which again contradicts \eq{sign}.

To complete the proof, finally we show that the function $F(x,y):=h_E(x/y)$ is a quasideviation on $\R_+^2$. The sign property of $F$ is obvious. The concavity of $h_E$ implies its continuity, hence $F$ is continuous in its second variable. Let $0<x<y$ and consider the functions
\Eq{*}{
   (x,y)\ni t\mapsto h_E(y/t)\qquad \mbox{and}\qquad 
   (x,y)\ni t\mapsto \frac{1}{-h_E(x/t)}.
}
Due to the monotonicity properties of $h_E$, it follows that the first is a positive nonincreasing function and the second one is a positive strictly decreasing function. Therefore, the ratio function 
\Eq{*}{
  \frac{F(y,t)}{F(x,t)}=-\frac{h_E(y/t)}{-h_E(x/t)}
}
is strictly increasing on $(x,y)$. In view of this quasideviation property, we have that $\LLD_F=\UUD_F=\D_F$, that is, $\LLE_{h_E}=\UUE_{h_E}=\E_{h_E}$.
On the other hand, by Theorem \ref{JC}, $E$ is also a quasideviation, hence the inequalities in \eq{EI} yield
\Eq{*}{
  \E_{h_E}=\LLE_{h_E}\leq\big(\UD_E\big)_\#=\big(\D_E\big)_\#
  \leq \big(\D_E\big)^\#=\big(\LD_E\big)^\#\leq \UUE_{h_E}=\E_{h_E}.
}
Notice that the most-left and most-right hand side coincide, whence it is in fact a sequence of equalities. This results assertion \eq{EI+}.

The Jensen concavity of $\D_E$ is a consequence of Theorem \ref{JC}. 
Using that $E^*$ is concave, it follows that it is nondecreasing in its second variable. (The proof is analogous to the proof of the nondecreasingness of $h_E$.)
\end{proof}

\section{Minkowski- and Hölder-type inequalities}

In what follows, we establish necessary and sufficient conditions for inequalities of Minkowski- and H\"older-type involving semideviation means. The proof is an elaborated adaptation of the methods from the papers by Dar\'oczy \cite{Dar71b} and P\'ales \cite{Pal88a}.

For the sake of convenience, given a function $f:J\times K\to I$ and two vectors $x=(x_1,\dots,x_n)\in J^n$, $y=(y_1,\dots,y_n)\in K^n$, we denote
\Eq{*}{
  f(x,y)=\big(f(x_1,y_1),\dots,f(x_n,y_n)\big).
} 

\begin{thm}
\label{thm:HoMiJe}
Let $I$, $J$, and $K$ be open intervals $E$, $F$, and $G$ be normalizable semideviations on the intervals $I$, $J$, and $K$, respectively. Let $f \colon J\times K \to I$ be a differentiable function. Then the following three assertions are equivalent:
\begin{enumerate}[(i)]
\item For all means $M\in\big\{\LLD_F,\LD_F,\UD_F,\UUD_F\big\}$ and 
$N\in\big\{\LLD_G,\LD_G,\UD_G,\UUD_G\big\}$, for all $n \in \N$, $x \in J^n$, $y\in K^n$, and $\lambda\in W_n(\R)$,
\Eq{add0}{
 \LLD_E\big(f(x,y),\lambda\big) 
   \le f\big(M(x,\lambda),\, N(y,\lambda)\big);
} 
\item There exist means $M\in\big\{\LLD_F,\LD_F,\UD_F,\UUD_F\big\}$ and 
$N\in\big\{\LLD_G,\LD_G,\UD_G,\UUD_G\big\}$ such that, for all $n \in \N$, $x \in J^n$, $y\in K^n$, and $\lambda\in W_n(\R)$, inequality \eq{add0} holds;
\item For all $p,u\in J$ and $q,v\in K$,
\Eq{add1}{
E^*\big(f(p,q),f(u,v)\big)
\le \partial_1f(u,v) F^*(p,u) +\partial_2f(u,v) G^*(q,v).
}
\end{enumerate}
\end{thm}

\begin{proof} The implication (i)$\Rightarrow$(ii) is obvious.
 
To prove (ii)$\Rightarrow$(iii), assume that \eq{add0} holds for some $M\in\big\{\LLD_F,\LD_F,\UD_F,\UUD_F\big\}$ and $N\in\big\{\LLD_G,\LD_G,\UD_G,\UUD_G\big\}$. Let $p,u\in J$ and $q,v\in K$. Then, by \eq{add0}, we have
\Eq{*}{
n \Big(\LLD_E\Big(\big(f(p,q),f(u,v)\big)&,(1,n)\Big)-f(u,v) \Big)\\
&\le 
n \Big(f\Big(M\big((p,u),(1,n)\big),N\big((q,v),(1,n)\big)-f(u,v)\Big),
}
thus
\Eq{*}{
&\lim_{n \to \infty} n \Big(\LLD_E\Big(\big(f(p,q),f(u,v)\big),(1,n)\Big)-f(u,v) \Big) \\
&\quad\le \partial_1f(u,v) \lim_{n \to \infty} n \big(M\big((p,u),(1,n)\big)-u\big) +\partial_2f(u,v) \lim_{n\to\infty} n \big(N\big((q,v),(1,n)\big)-v\big).
}
Computing the limits with the help of Lemma~\ref{Lim}, we obtain that \eq{add1} is valid.

To show that the implication (iii)$\Rightarrow$(i) is also valid, assume that \eq{add1} holds. Let $M\in\big\{\LLD_F,\LD_F,\UD_F,\UUD_F\big\}$ and $N\in\big\{\LLD_G,\LD_G,\UD_G,\UUD_G\big\}$ be arbitrary and let $n \in \N$, $x=(x_1,\dots,x_n) \in J^n$, $y=(y_1,\dots,y_n)\in K^n$, and $\lambda=(\lambda_1,\dots,\lambda_n) \in W_n(\R)$.
Define
\Eq{*}{
e_{E}(t):=\sum_{i=1}^n \lambda_i E^*(f(x_i,y_i),t), \qquad
e_{F}(t):=\sum_{i=1}^n \lambda_i F^*(x_i,t), \quad e_{G}(t):=\sum_{i=1}^n \lambda_i G^*(y_i,t),
}
for $t$ in $I$, $J$, and $K$, respectively.
Now, if we substitute $p:=x_i$ and $q:=y_i$ in inequality \eq{add1}, multiply it by $\lambda_i \ge 0$ and sum up the inequalities side by side for all $i\in \{1,\dots,n\}$, we obtain, for all $(u,v)\in J\times K$,
\Eq{uv}{
e_{E}(f(u,v)) \le \partial_1f(u,v) e_{F}(u) +\partial_2f(u,v) e_{G}(v).
}
Now putting $u=M(x,\lambda)$ and $v=N(y,\lambda)$, we have that $e_{F}(u)=e_{G}(v)=0$, hence the inequality $e_{E}(f(u,v)) \le 0$ is valid. It implies that $\LLD_E(f(x,y),\lambda)\leq f(u,v)$, which was to be proved.
\end{proof}

In the case, when the operation $f:J\times K\to I$ is increasing in each of its variables, we have the following more precise statement.

\begin{thm}
\label{thm:HoMi}
Let $I$, $J$, and $K$ be open intervals $E$, $F$, and $G$ be normalizable semideviations on the intervals $I$, $J$, and $K$, respectively. Let $f \colon J\times K \to I$ be a differentiable function such that $\partial_1f,\partial_2f\geq0$ and $\partial_1f+\partial_2f>0$ on $J\times K$. Then the following assertions are equivalent:
\begin{enumerate}[(i)]
\item For all $n \in \N$, $x \in J^n$, $y\in K^n$, and $\lambda\in W_n(\R)$,
\Eq{*}{
 \LLD_E\big(f(x,y),\lambda\big) 
   \le f\big(\LLD_F(x,\lambda),\, \LLD_G(y,\lambda)\big);
} 
\item For all $n \in \N$, $x \in J^n$, $y\in K^n$, and $\lambda\in W_n(\R)$,
\Eq{*}{
 \LD_E\big(f(x,y),\lambda\big) 
   \le f\big(\LD_F(x,\lambda),\, \LD_G(y,\lambda)\big);
}
\item For all $n \in \N$, $x \in J^n$, $y\in K^n$, and $\lambda\in W_n(\R)$,
\Eq{*}{
 \UD_E\big(f(x,y),\lambda\big) 
   \le f\big(\UD_F(x,\lambda),\, \UD_G(y,\lambda)\big);
} 
\item For all $n \in \N$, $x \in J^n$, $y\in K^n$, and $\lambda\in W_n(\R)$,
\Eq{*}{
 \UUD_E\big(f(x,y),\lambda\big) 
   \le f\big(\UUD_F(x,\lambda),\, \UUD_G(y,\lambda)\big);
}
\item For all $n \in \N$, $x \in J^n$, $y\in K^n$, and $\lambda\in W_n(\R)$,
\Eq{*}{
 \LLD_E\big(f(x,y),\lambda\big) 
   \le f\big(\UUD_F(x,\lambda),\, \UUD_G(y,\lambda)\big);
}
\item For all $p,u\in J$ and $q,v\in K$, inequality \eq{add1} holds.
\end{enumerate}
\end{thm}

\begin{proof} Proposition \ref{prop:SDM} and the monotonicity property of $f$ yield that each of the assertions (i), (ii), (iii), and (iv) implies assertion (v).

The implication (v)$\Rightarrow$(vi) is a consequence of implication (ii)$\Rightarrow$(iii) of the previous theorem.

It remains to prove that each of the assertions (i), (ii), (iii), and (iv) is a consequence of (vi). From now on, we assume that assertions (vi) holds.

In the rest of the proof let $n\in\N$, $x=(x_1,\dots,x_n) \in J^n$, $y=(y_1,\dots,y_n)\in K^n$, and $\lambda=(\lambda_1,\dots,\lambda_n) \in W_n(\R)$ and define the functions $e_E$, $e_F$, and $e_G$ as in the proof of the previous theorem. Then, repeating the same argument as therein, we get the inequality \eq{uv} for all $(u,v)\in J\times K$.

To see that (i) holds, let $u_0:=\LLD_F(x,\lambda)$ and $v_0:=\LLD_G(y,\lambda)$. Then $e_F(u_0)=0$ and $e_G(v_0)=0$, therefore, by inequality \eq{uv}, we get  $e_E(f(u_0,v_0))\leq0$ and hence $\LLD_E(f(x,y),\lambda)\leq f(u_0,v_0)$.

To see that (ii) is valid, let $u_0:=\LD_F(x,\lambda)$ and $v_0:=\LD_G(y,\lambda)$. Then there exist sequences $u_n>u_0$ and $v_n>v_0$ converging to $u_0$ and $v_0$, respectively, such that $e_F(u_n)<0$ and $e_G(v_n)<0$ for all $n\in\N$.  Hence, inequality \eq{uv} and the sign properties of the partial derivatives of $f$ imply that $e_E(f(u_n,v_n))<0$ for all $n\in\N$. Then, we have that $\LD_E(f(x,y),\lambda)\leq f(u_n,v_n)$ for all $n$. Passing the limit $n\to\infty$, it follows that
$\LD_E(f(x,y),\lambda)\leq f(u_0,v_0)$.

To show that (iii) is true, denote $u_0:=\UD_F(x,\lambda)$ and $v_0:=\UD_G(y,\lambda)$. Then, for all $u>u_0$ and $v>v_0$, we get that $e_F(u)\leq0$ and $e_G(v)\leq0$. Hence, inequality \eq{uv} and the sign properties of the partial derivatives of $f$ imply that $e_E(f(u,v))\leq0$ for all $u>u_0$ and $v>v_0$. Choose $\alpha\in\{1,\dots,n\}$ so that $f(x_\alpha,y_\alpha)=\max_{1\leq i\leq n} f(x_i,y_i)$. Then, for $t>f(x_\alpha,y_\alpha)$, we have that $e_E(t)<0$. In the case when $t\in(f(u_0,v_0),f(x_\alpha,y_\alpha)]$, there exist $u\in(u_0,x_\alpha]$ and $v\in(v_0,y_\alpha]$ such that $t=f(u,v)$. Hence, we have proved that $e_E(t)\leq 0$ also for $t>f(u_0,v_0)$. This implies that $\UD_E(x,\lambda)\leq f(u_0,v_0)$.

To verify that (iv) is valid, denote $u_0:=\UUD_F(x,\lambda)$ and $v_0:=\UUD_G(y,\lambda)$. Then, for all $u>u_0$ and $v>v_0$, we get that $e_F(u)<0$ and $e_G(v)<0$. Hence, inequality \eq{uv} and the sign properties of the partial derivatives of $f$ imply that $e_E(f(u,v))<0$ for all $u>u_0$ and $v>v_0$. Choose $\alpha\in\{1,\dots,n\}$ so that $f(x_\alpha,y_\alpha)=\max_{1\leq i\leq n} f(x_i,y_i)$. Then, for $t>f(x_\alpha,y_\alpha)$, we have that $e_E(t)<0$. In the case when $t\in(f(u_0,v_0),f(x_\alpha,y_\alpha)]$, there exist $u\in(u_0,x_\alpha]$ and $v\in(v_0,y_\alpha]$ such that $t=f(u,v)$. Hence, we have proved that $e_E(t)<0$ also for $t>f(u_0,v_0)$. This implies that $\UUD_E(x,\lambda)\leq f(u_0,v_0)$.
\end{proof}

In what follows we present some special cases of the previous theorem. Namely, when $f$ is of the form $f(x,y)=x$, then the above result reduces to Theorem \ref{com}. Taking $f$ to be the arithmetic mean and $E=F=G$, the equivalence of statements (i), (ii), and (iii) of Theorem \ref{JC} follows also from the above theorem. If $f$ is either the addition or multiplication of the two variables, then we can deduce characterizations of Minkowski- and H\"older-type inequalities that are stated in the following two corollaries.

\begin{cor}
Let $J$ and $K$ be open intervals $E$, $F$, and $G$ be normalizable semideviations on the intervals $I:=J+K$, $J$, and $K$, respectively. Then the following assertions are equivalent:
\begin{enumerate}[(i)]
\item For all $n \in \N$, $x \in J^n$, $y\in K^n$, and $\lambda\in W_n(\R)$,
\Eq{*}{
 \LLD_E(x+y,\lambda) 
   \le \LLD_F(x,\lambda) + \LLD_G(y,\lambda);
} 
\item For all $n \in \N$, $x \in J^n$, $y\in K^n$, and $\lambda\in W_n(\R)$,
\Eq{*}{
 \LD_E(x+y,\lambda) 
   \le \LD_F(x,\lambda) + \LD_G(y,\lambda);
} 
\item For all $n \in \N$, $x \in J^n$, $y\in K^n$, and $\lambda\in W_n(\R)$,
\Eq{*}{
 \UD_E(x+y,\lambda) 
   \le \UD_F(x,\lambda) + \UD_G(y,\lambda);
} 
\item For all $n \in \N$, $x \in J^n$, $y\in K^n$, and $\lambda\in W_n(\R)$,
\Eq{*}{
 \UUD_E(x+y,\lambda) 
   \le \UUD_F(x,\lambda) + \UUD_G(y,\lambda);
} 
\item For all $n \in \N$, $x \in J^n$, $y\in K^n$, and $\lambda\in W_n(\R)$,
\Eq{*}{
 \LLD_E(x+y,\lambda) 
   \le \UUD_F(x,\lambda) + \UUD_G(y,\lambda);
} 
\item For all $p,u\in J$ and $q,v\in K$,
\Eq{*}{
E^*\big(p+q,u+v\big)
\le F^*(p,u) + G^*(q,v).
}
\end{enumerate}
\end{cor}

\begin{cor}
Let $J,\,K \subset \R_+$ be open intervals $E$, $F$, and $G$ be normalizable semideviations on the intervals $I:=J \cdot K$, $J$, and $K$, respectively. Then the following assertions are equivalent:
\begin{enumerate}[(i)]
\item For all $n \in \N$, $x \in J^n$, $y\in K^n$, and $\lambda\in W_n(\R)$,
\Eq{*}{
 \LLD_E(x\cdot y,\lambda) 
   \le \LLD_F(x,\lambda) \cdot \LLD_G(y,\lambda);
} 
\item For all $n \in \N$, $x \in J^n$, $y\in K^n$, and $\lambda\in W_n(\R)$,
\Eq{*}{
 \LD_E(x\cdot y,\lambda) 
   \le \LD_F(x,\lambda) \cdot \LD_G(y,\lambda);
} 
\item For all $n \in \N$, $x \in J^n$, $y\in K^n$, and $\lambda\in W_n(\R)$,
\Eq{*}{
 \UD_E(x\cdot y,\lambda) 
   \le \UD_F(x,\lambda) \cdot \UD_G(y,\lambda);
} 
\item For all $n \in \N$, $x \in J^n$, $y\in K^n$, and $\lambda\in W_n(\R)$,
\Eq{*}{
 \UUD_E(x\cdot y,\lambda) 
   \le \UUD_F(x,\lambda) \cdot \UUD_G(y,\lambda);
} 
\item For all $n \in \N$, $x \in J^n$, $y\in K^n$, and $\lambda\in W_n(\R)$,
\Eq{*}{
 \LLD_E(x\cdot y,\lambda) 
   \le \UUD_F(x,\lambda) \cdot \UUD_G(y,\lambda);
}
\item For all $p,u\in J$ and $q,v\in K$,
\Eq{*}{
E^*\big(pq,uv\big)
\le v\, F^*(p,u) +u\, G^*(q,v).
}
\end{enumerate}
\end{cor}

%\bibliography{publ,funcequ,newbib}
%\bibliographystyle{plain}

\end{document}